\documentclass{amsart}
\usepackage{amsfonts,amsmath,amsthm,amssymb, amscd}

\newtheorem{theorem}{Theorem}
\newtheorem{lemma}{Lemma}

\def\m1{^{-1}}

\DeclareMathOperator{\rad}{\mathfrak{rad}}

\begin{document}

\title{Filtered multiplicative bases of restricted enveloping algebras}
\author{V.~BOVDI, A.~GRISHKOV, S.~SICILIANO}

\address{
\texttt{VICTOR BOVDI},
\newline
University of Debrecen,
H--4010 Debrecen, P.O. BOX 12, Hungary}
\email{vbovdi@math.klte.hu}

\address{
\texttt{ALEXANDER GRISHKOV}
\newline
IME, USP, Rua do Matao, 1010 -- Citade Universit\`{a}ria, CEP
05508-090, Sao Paulo, Brazil} \email{shuragri@gmail.com}

\address{
\texttt{SALVATORE SICILIANO},
\newline
Dipartimento di Matematica ``E. De Giorgi",
Universit\`{a} del Salento,
Via Provinciale Lecce--Arnesano, 73100--LECCE, Italy}
\email{salvatore.siciliano@unile.it}


\thanks{The research was supported by OTKA  No.K68383, RFFI 07-01-00392A, FAPESP and CNPq(Brazil)}

\begin{abstract}
We study the problem of the existence of filtered multiplicative
bases of a restricted enveloping algebra $u(L)$, where $L$
is a finite-dimensional and $p$-nilpotent restricted Lie
algebra over a field of positive characteristic $p$.
\end{abstract}

\maketitle{}

\section{Introduction  and results}

Let $A$ be a finite--dimensional associative algebra over a field
$F$. Denote by $\rad(A)$ the Jacobson radical of $A$ and let
$\mathfrak{B}$ be an $F$--basis of $A$. Then $\mathfrak{B}$ is
called a \emph{filtered multiplicative basis} (f.m. basis) of $A$ if the
following properties hold:

\begin{enumerate}
\item [(i)] for every $b_1,b_2 \in \mathfrak{B}$ either $b_1b_2=0$
or $b_1b_2 \in \mathfrak{B}$;
\item [(ii)] $\mathfrak{B}\cap\rad(A)$ is an $F$--basis of $\rad(A)$.
\end{enumerate}
Filtered multiplicative bases   arise in the theory of
representation of associative algebras and were introduced by
H.~Kupisch in \cite{Kupisch_I}. In their celebrated paper
\cite{Bautista_Salmeron} R.~Bautista, P.~Gabriel, A.~Roiter and
L.~Salmeron proved that if $A$ has finite representation type
(that is, there are only finitely many isomorphism classes of
finite-dimensional indecomposable $A$--modules) over an algebraically closed
field $F$, then $A$ has an f.m. basis.

In \cite{Gabriel_Roiter} an analogous statement was proposed for
finitely spaced modules over an aggregate. (Such  modules
give rise to a matrix problem in which the allowed column
transformations are determined by the  module structure, the row
transformations are arbitrary, and the number of canonical
matrices is finite).  This statement was subsequently proved in \cite{Roiter_Sergeichuk}.

The problem of existence of an f.m. basis in a group algebra was
posed in \cite{Bautista_Salmeron} and  has been considered by
several authors: see e.g. \cite{Balogh_II, Balogh_I, Bovdi_I,Bovdi_II, Bovdi_IE,
Landrock_Michler}. In particular, it is still an open problem whether
a group algebra $KG$  has  an f.m. basis in the case when  $F$
is a field  of odd characteristic $p$ and $G$ is a nonabelian
$p$-group.

Apparently, not much is known about the same problem in the setting of
restricted enveloping algebras.
The present paper  represents  a contribution in this  direction.
In particular, because of the analogy with the theory of
finite $p$-groups, we confine our attention to the class
$\mathfrak{F}_p$ of finite--dimensional and $p$--nilpotent
restricted Lie algebras over a field of positive characteristic $p$.
Note that under this assumption, the
aforementioned result in \cite{Bautista_Salmeron} can be applied  only in
very special cases. Indeed, for $L\in \mathfrak{F}_p$,  from
\cite{Feldvoss_Strade} it follows that $u(L)$  has finite
representation type if and only if $L$ is cyclic, that is, there
exists an element which generates $L$ as a restricted subalgebra.
\newpage

Our main results are the following three theorems.
\begin{theorem}\label{T:1}
Let $L\in \mathfrak{F}_p$ be an
abelian restricted Lie algebra over a field $F$. Then $u(L)$
has a filtered multiplicative basis if and only if $L$ decomposes
as a direct sum of cyclic restricted subalgebras. In particular,
if $F$ is a perfect field, then $u(L)$ has a filtered
multiplicative basis.
\end{theorem}

A restricted Lie algebra ${\mathfrak L}\in \mathfrak{F}_p$
is called  \emph{powerful} (see e.g. \cite{Siciliano_Weigel}) if $p=2$ and
${\mathfrak L}^\prime \subseteq {\mathfrak L}^{[p]^2}$ or
$p>2$ and ${\mathfrak L}^\prime \subseteq {\mathfrak L}^{[p]}$.
Here ${\mathfrak L}^{[p]^i}$ denotes the restricted subalgebra
generated by the elements $x^{[p]^i}$, $x \in L$.

\begin{theorem}\label{T:2}
Let $L\in \mathfrak{F}_p$ be a nonabelian restricted Lie
algebra over a field $F$. If $L$ is powerful then $u(L)$ does
not have a filtered multiplicative basis.
\end{theorem}

\begin{theorem}\label{T:3}
Let $L\in \mathfrak{F}_p$ be a restricted Lie algebra over a
field $F$. If $L$ has nilpotency class 2 and $p>2$ then $u(L)$ does not have a filtered
multiplicative basis.
\end{theorem}
An example showing that Theorem \ref{T:3} fails in characteristic 2 is also provided.
Finally, we remark that for odd $p$
no example of noncommutative restricted enveloping algebra
having an f.m. basis seems to be known.

\medskip

\section{Preliminaries}

Let $A$ be a finite--dimensional associative algebra over a field
$F$. If $\mathfrak{B}$ is an f.m. basis of $A$ then the following
simple properties hold (see \cite{Bovdi_I}):
\begin{itemize}
\item[(F-I)] $\mathfrak{B}\cap \rad(A)^n$ is an $F$--basis of
$\rad(A)^n$ for every $n\geq 1$;
\item[(F-II)] if $u,v\in \mathfrak{B} \backslash \rad(A)^k$ and
$u\equiv v\pmod{\rad(A)^k}$
then $u=v$.
\end{itemize}

Let $L$ be a restricted Lie algebra over a field $F$ of
characteristic $p>0$ with a $p$--map $[p]$. We denote by $\omega(L)$ the
\emph{augmentation ideal} of $u(L)$, that is, the associative
ideal generated by $L$ in $u(L)$.  In \cite{Riley_Shalev_II},
the \emph{ dimension subalgebras} of $L$ were defined as  the
restricted ideals of $L$ given by
$$
\mathfrak{D}_m(L)=L\cap \omega(L)^m \qquad \quad (m\geq 1).
$$
These subalgebras can be  explicitly described as
$\mathfrak{D}_m(L)=\sum_{ip^{j}\geq m}{{\gamma}_{i}(L)}^{[p]^{j}}$,
where ${\gamma}_{i}(L)^{[p]^j}$ is the restricted subalgebra of
$L$ generated by the set of $p^j$--th powers of the $i$--th term
of the lower central series of $L$. Note that $\mathfrak{D}_2(L)$
coincides with the Frattini restricted subalgebra $\Phi(L)$ of $L$ (cf.
\cite{Lincoln_Towers}).

It is well-known that if $L$ is finite--dimensional and
$p$--nilpotent then $\omega(L)$
is nilpotent.
Clearly, in this case $\omega(L)$  coincides with  $\rad(u(L))$ and
$u(L)=F\cdot 1\oplus \omega(L)$. Consequently, if $u(L)$ has
an  f.m. basis $\mathfrak{B}$, then we can assume without loss of
generality that $1\in \mathfrak{B}$. For each $x\in L$, the
largest subscript $m$ such that $x \in \mathfrak{D}_m(L)$ is
called the \emph{height} of $x$ and is denoted by $\nu(x)$. The
combination of Theorem $2.1$ and Theorem $2.3$ from
\cite{Riley_Shalev_II} gives the following.

\begin{lemma}\label{L:1}
Let $L\in \mathfrak{F}_p$ be a restricted Lie algebra over a field $F$,
and let $\{x_i \}_{i\in I}$ be an ordered basis  of $L$ chosen such  that
$$
\mathfrak{D}_m(L)=\mathrm{span}_F \{ x_i\;  \vert\;  \nu(x_i)\geq m\}\qquad  (m\geq 1).
$$
Then for each positive integer $n$ the following statements hold:
\begin{enumerate}
\item [(i)] $\omega(L)^n=\mathrm{span}_F \{x\; \vert \;\nu(x)\geq n \}$, where\quad  $x=x^{\alpha_1}_{i_1}\cdots
x^{\alpha_l}_{i_l}$,\newline
$\nu(x)=\sum_{j=1}^l\alpha_j\nu(x_{i_j})$, \quad $i_1<\cdots
<i_l$\quad and\quad  $0\leq \alpha_j \leq p-1$. \item [(ii)] The
set $\{\; y \; \vert\;  \nu(y)=n \;\}$ is an $F$--basis of
$\omega(L)^n$ modulo $\omega(L)^{n+1}$.
\end{enumerate}
\end{lemma}

For a subset $S$ of $L$ we shall denote by $\langle S \rangle_p$
the restricted subalgebra generated by $S$. Moreover, if $z\in L$
is $p$--nilpotent, the minimal positive
integer $n$ such that $z^{[p]^n}=0$ is called the
\emph{exponent} of $z$ and denoted by $e(z)$.

\medskip
\section{Proofs}

\begin{proof}[Proof of the Theorem 1.] Assume first that
$L=\bigoplus_{i=1}^n \langle x_i \rangle_p$. Then, by the PBW
Theorem for restricted Lie algebras (see
\cite{Strade_Farnsteiner}, Chapter $2$, Theorem $5.1$), we see that $u(L)$
is isomorphic to the truncated polynomial algebra
\[
F[X_1,\ldots,X_n]/(X_1^{p^{e(x_1)}},\ldots,X_n^{p^{e(x_n)}}).
\]
Consequently  the algebra $u(L)$ has an f.m. basis.

Conversely, suppose that $u(L)$ has an f.m. basis $\mathfrak{B}_1$
with $1\in \mathfrak{B}_1$ and put
$\mathfrak{B}=\mathfrak{B}_1\backslash \{1\}$. Let $n=\dim_F L/L^{[p]}$ and
$\mathfrak{B}\backslash  \omega(L)^2=\{b_1,\ldots,b_n\}$.
By Lemma \ref{L:1}, one has $b_i=x_i+h_i$, where $x_i\in L\backslash L^{[p]}$ and
$h_i\in \omega(L)^2$ for every $i=1,\ldots,n$. From
\cite{Lincoln_Towers} it follows at once that $\{x_1,\ldots, x_n\}$
is a minimal set of generators of $L$ as a restricted
subalgebra. We shall prove by induction on
$e=\max \{i \vert\, L^{[p]^i} \neq 0\}$ that $L$ has a cyclic decomposition.
If $e=1$, then $L= \bigoplus_{i=1}^n \langle x_i \rangle_p$.

Now let $e>1$ and suppose that $L$ does not decompose as a direct sum of
restricted subalgebras. Since $L= \sum_{i=1}^n \langle x_i \rangle_p$ and the $p$-map is
$p$-semilinear, note that $L^{[p]^e}$ is just the vector subspace generated by the $p^e$--th powers
of the generators $x_i$ having exponent $e+1$. Therefore, without loss  of generality we can assume that
$$
e+1=e(x_1)=\dots=e(x_m)\geq e(x_{s}),\qquad\quad (m+1\leq s\leq  n)
$$
and  $\{x_1^{[p]^e},\dots,x_m^{[p]^e}\}$ is an $F$-linearly independent set with
$$
\mathrm{span}_F \{x_1^{[p]^e},\dots,x_m^{[p]^e}\}= L^{[p]^e}.
$$
In turn, we can reindex  the elements $x_{m+1},\dots,x_n$ so that there exists
a maximal $m\leq k< n$ such that
$$
H=\langle x_1,\ldots,x_m,\ldots,x_k \rangle_p = \bigoplus_{i=1}^k \langle y_i \rangle_p
$$
for suitable $y_{1},\ldots,y_k$ in $L$ with $y_i=x_i$ for $i=1,\ldots,m$.
Consequently, for every $s>k$ there exists a minimal number $f_s$ such that
\begin{equation}\label{E:1}
x_s^{p^{f_s}}=\sum_{i=1}^{k} \sum_{j=0}^{e(y_i)-1} \mu^{(s)}_{i,j}y_i^{[p]^j}\qquad\quad  (\mu^{(s)}_{i,j}\in F).
\end{equation}
Denote by $J$ the associative ideal of $u(L)$ generated by the elements $b_1^{p^e},\ldots,b_m^{p^e}$.
Clearly $J\subseteq \omega(L)^{p^e}\subseteq L^{[p]^e}u(L)$. Suppose by
contradiction that $J\neq L^{[p]^e}u(L)$. If  $r$ is  the maximal positive integer such that
$$
\left( L^{[p]^e}u(L)\cap \omega(L)^r \right) \backslash J \neq \emptyset,
$$
then  there exists $v=x_i^{p^e}x_1^{a_1}\cdots x_n^{a_n}\in \omega(L)^r\setminus J$ such that
$$
v \equiv b_i^{p^e}b_1^{a_1}\cdots b_n^{a_n}\pmod{\omega(L)^{r+1}}.
$$
Consequently
$$
v-b_i^{p^e}b_1^{a_1}\cdots b_n^{a_n}\in \left( L^{[p]^e}u(L)\cap \omega(L)^{r+1} \right) \backslash J,
$$
contradicting  the definition of $r$.
Therefore  $J=L^{[p]^e}u(L)$, which implies that  $u(L)/J \cong u(L/L^{[p]^e})$. Moreover, it is easily seen that $\mathfrak{B}_1\cap J$ is an $F$-basis of $J$,  hence the elements $b_i+J$ with $b_i\notin J$ form an f.m. basis of $u(L)/J$. Consequently, by induction we have that $\mathfrak{L}= L/L^{[p]^e}$ is a direct sum of restricted subalgebras.

As the images of $y_1,\ldots,y_k$ are $F$-linearly independent in $\mathfrak{L}/\Phi(\mathfrak{L})$,
from \cite{Lincoln_Towers} it follows that there exists a restricted subalgebra  $P$ of $L$ with $L^{[p]^e}\subseteq P$  such that
$$
\mathfrak{L}=H/L^{[p]^e}\oplus P/L^{[p]^e}.
$$
As a consequence, for every $s>k$ we have
$ x_s\equiv v_s + w_s \pmod{L^{[p]^e}}
$ with $v_s\in H$ and $w_s\in P$ and, moreover, it follows from (\ref{E:1})  that  $w_s^{[p]^{f_s}}\in L^{[p]^e}$.
One has
\begin{equation}\label{R:2}
v_s=\sum_{i=1}^{k} \sum_{j=0}^{e(y_i)-1} k_{i,j}^{(s)}y_i^{[p]^j},\qquad \qquad (k_{i,j}^{(s)} \in F).
\end{equation}
Since $x_s^{[p]^{f_s}} \equiv v_s^{[p]^{f_s}}\pmod{L^{[p]^e}}$,
we conclude that $\mu^{(s)}_{i,j}\in F^{p^{f_s}}$ provided $j<e$.
 We claim that  for every $1\leq i\leq k$ the coefficient  $\mu^{( s)}_{i,e}$ is also in $F^{p^{f_s}}$.
Indeed,  write
$$
w_{ s}=\sum_{b\in \mathfrak{B}} \lambda_b b\qquad (\text{for suitable}\quad  \lambda_b\in F).
$$
Then, as $\mathfrak{B}$ is a filtered multiplicative basis of $u(L)$, it follows that
\begin{equation}\label{R:1}
w_{ s}^{p^{f_s}}=\sum_{b\in \mathfrak{C}} \mu_b^{p^{f_s}} b^{p^{f_s}}
\end{equation}
where $ \mathfrak{C}$ is a subset of  $\mathfrak{B}$ and the $\mu_b$'s are nonzero elements of $F$.
Moreover, since $w_{s}^{[p]^{f_s}}\in L^{[p]^e}$ we have
\begin{equation}\label{S:1}
w_{ s}^{p^{f_s}}=\sum_{i=1}^m \alpha_i b_i^{[p]^e}.
\end{equation}
As $\mathfrak{B}$ is a filtered $F$-basis of $u(L)$, by comparing (\ref{R:1})
and (\ref{S:1}) we conclude that for every $i=1,2,\ldots,m$ there exists $\beta_i\in F$
such that $\alpha_i=\beta_i^{p^{f_s}}$.
At this stage, the relations (\ref{E:1}) and (\ref{R:2}) allows to conclude that for every $1\leq i\leq k$ one has
$$
 \mu^{(s)}_{i,e}= \left( k_{i,e}^{(s)}\right)^{p^{f_s}}+ \beta_i^{p^{f_s}} \in F^{p^{f_s}},
$$
as desiderate (here $\beta_{m+1},\ldots, \beta_{k}=0$).

Now, by (\ref{E:1}) and the above discussion we have  $x_{s}^{p^{f_s}}=z^{p^{f_s}}$ for some $z\in H$. Therefore $(x_{\bar s}-z)^{p^{f_s}}=0$ and then the minimality of $f_s$ forces $\langle x_{\bar s}-z\rangle_p\cap H=0$. This contradicts the definition of $k$, yielding the claim.

Finally, if $F$ is perfect, then $L$ decomposes as a direct sum of cyclic
restricted subalgebras (see e.g. \cite{Bahturin_Zaicev}, Chapter
4, Theorem in Section 3.1). The proof is done.
\end{proof}

\newpage
Unlike   group algebras,  a commutative restricted enveloping
algebra need not have an f.m. basis. Indeed, we have the following

\smallskip
{\bf Example.} Let $F$ be a field of positive
characteristic $p$ containing an element $\alpha$
which is not a $p$-th root in $F$. Consider the abelian
restricted Lie algebra
$$
L_\alpha= Fx+Fy + Fz
$$
with $x^{[p]}=\alpha z$,
$y^{[p]}=z$, and $z^{[p]}=0$.
Suppose  that $u(L_\alpha)$ has an f.m. basis. By
Theorem \ref{T:1}, $L_\alpha$ is a direct sum of cyclic restricted
subalgebras. Since $L_\alpha^{[p]}\neq 0$ and
$L_\alpha^{[p]^2}=0$, we have
$
L_\alpha=\langle a \rangle_p \oplus \langle b \rangle_p$
with $e(a)=2$ and $e(b)=1$.

Let  $b=k_1x+k_2y+k_3z$, $k_i\in F$. Since $\alpha \notin
F^p$, we get $0=b^{[p]}=(k_1^p \alpha+k_2^p)z$,  so  $k_1=k_2=0$
and  $0\neq a^{[p]}\in Fz =\langle b \rangle_p$, a contradiction.

\begin{lemma}\label{L:2}
Let $A$ be a finite-dimensional nilpotent associative algebra over a field $F$. Suppose that
$A$ has a minimal set of generators $\{u_1,\ldots,u_n \}$ such that:
\begin{itemize}
\item [(i)]\quad  $[u_i,u_j]\in A^3$ \quad for every\quad $i,j=1,\ldots,n$;
\item[(ii)] \quad $u_iu_j\notin A^3$ \quad for every\quad $i,j=1,\ldots,n$;
\item [(iii)]\quad  $\emph{span}_F\{u_iu_j \vert
1\leq i<j \leq n\} \cap \emph{span}_F \{u_i^2 \vert i=1,\ldots,n \}
\subseteq A^3$.
\end{itemize}
Then $A$ has no f.m. basis.
\end{lemma}
\begin{proof} By contradiction, assume that there exists  an f.m. basis
$\mathfrak{B}$ of $A$. Clearly, we have $\dim_F A/A^2=n$ and, by
property (F-I), $\mathfrak{B}\backslash  \omega(L)^2$ is a minimal
set of generators of $A$ as an associative algebra. Write
$\mathfrak{B}\backslash  A^2=\{b_1,\ldots,b_n\}$. Obviously
$$
b_k \equiv \sum_{i=1}^n \alpha_{ki}u_i \pmod{A^2}, \qquad
\qquad(\alpha_{ki} \in F)
$$
 and  the determinant of the matrix $M=(a_{ki})$ is not zero. Now
\[
\begin{split}
b_r b_s \equiv \sum_{i=1}^n \alpha_{ri}\alpha_{si}u_{i}^2&+
\sum_{\begin{subarray}{c}i,j=1\\  i<j \end{subarray}}^n
(\alpha_{ri}\alpha_{sj}+\alpha_{rj}\alpha_{si})u_{i}u_{j}\\
&-\sum_{\begin{subarray}{c}i,j=1\\  i<j \end{subarray}}^n \alpha_{rj}\alpha_{si}[u_i,u_j]
\pmod{A^3}.
\end{split}
\]
By  assumption (i) of the statement we have that $[u_i, u_j] \equiv 0 \pmod{A^3}$, so
\begin{equation}\label{E:4}
\begin{split}
b_r b_s &\equiv
\sum_{i=1}^n \alpha_{ri}\alpha_{si}u_{i}^2\\
&+\sum_{\begin{subarray}{c}i,j=1\\  i<j \end{subarray}}^n
(\alpha_{ri}\alpha_{sj}+\alpha_{rj}\alpha_{si})u_{i}u_{j}
 \equiv b_sb_r \pmod{A^3}.
 \end{split}
\end{equation}
Suppose $b_rb_s \in A^3$ for some $r,s$.  Because of (\ref{E:4}) and the assumptions (ii) and (iii) of the statement we have  $\alpha_{ri}\alpha_{si}=0$ and
$\alpha_{ri}\alpha_{sj}+\alpha_{si}\alpha_{rj}=0$ for every
$i,j$. It follows that $\alpha_{ri}\alpha_{sj}-\alpha_{si}\alpha_{rj}=0$.
Consequently, all of the order two minors  formed by the $k$-th and $s$-th
lines of the matrix $M$ are zero, which is impossible
because $\det M \not=0$. Hence\quad  $b_r b_s, b_s b_r \notin \omega(L)^3$
\; and\;  $b_r b_s \equiv b_s b_r \pmod{\omega(L)^3}$ \quad for every $r,s$. By
property (F-II) of the f.m. bases we conclude that\;  $b_r b_s = b_s
b_r$. Thus $A$ is a commutative algebra, a contradiction.
\end{proof}

\begin{proof}[Proof of the Theorem 2.] Let $S$ be a minimal set
of generators of $L$ as a restricted Lie algebra. Then, as $L$ is
powerful, by Lemma \ref{L:1} we conclude that $S$ is a minimal set of the nilpotent
associative algebra $\omega(L)$ satisfying the hypotheses of Lemma
\ref{L:2}, and the claim follows.
\end{proof}

\begin{proof}[Proof of the Theorem 3.] Suppose, by contradiction,
that $u(L)$ has an f.m. basis $\mathfrak{B}_1$ with $1\in \mathfrak{B}_1$, so that $\mathfrak{B}=\mathfrak{B}_1\backslash \{1\}$ is an f.m. basis of $\omega(L)=\rad(u(L))$. Put
$n=\dim_F\mathfrak{D}_1(L)/\mathfrak{D}_2(L)$ and write $\mathfrak{B}\backslash
\omega(L)^2=\{b_1,\ldots,b_n\}$. Consider  an $F$--basis $B$ of
$L$ as in the statement of Lemma  \ref{L:1} and let
$u_1,\ldots,u_n$ be the elements of $B$ having height $1$. Thus,
by Lemma \ref{L:1}(ii), the set\quad
$
\{u_j+\omega(L)^2  \vert\,  j=1,\ldots,n\}
$\quad
forms an $F$--basis of $\omega(L)/\omega(L)^2$. Then, for every
$k=1,\ldots,n$ there exist $\alpha_{k1},\ldots,\alpha_{kn} \in
F$ such that
$$
b_k \equiv \sum_{i=1}^n \alpha_{ki}u_i  \pmod{\omega(L)^2},\qquad\qquad (k=1,\ldots,n).
$$
Set $\bar{u}_k=\sum_{i=1}^n \alpha_{ki}u_i$. Plainly,
$\{\bar{u}_1,\ldots,\bar{u}_n \}$ is an $F$-linearly independent
set which generates $L$ as a restricted subalgebra.

Now, if $L$ is powerful then, by Theorem  \ref{T:2}, $u(L)$ cannot
have any f.m. basis, a contradiction. Therefore $L^\prime
\nsubseteq L^{[p]}$ and so there exist $1\leq r< s \leq n$ such
that the element $c_{rs}=[\bar{u}_r,\bar{u}_s]$ is not in
$L^{[p]}$. Since $L$ is nilpotent of class 2, we have that
$$
\mathfrak{D}_2(L)= L^\prime+L^{[p]} \supset L^{[p]}=\mathfrak{D}_3(L),
$$
hence  $c_{rs}$ has height two. Furthermore, one has
\[
\begin{split}
b_s^2 &\equiv \bar{u}^2_s \pmod{\omega(L)^3};  \qquad
b_sb_r \equiv \bar{u}_r \bar{u}_s - c_{rs} \pmod{\omega(L)^3}.
\end{split}
\]
Since $L$ is nilpotent of class 2, it follows that
\[
\begin{split}
b_rb_s^2 &\equiv \bar{u}_r \bar{u}^2_s \pmod{\omega(L)^4};\\
b_su_rb_s &\equiv \bar{u}_r \bar{u}^2_s - \bar{u}_sc_{rs}
\pmod{\omega(L)^4};\\
b_s^2 b_r&\equiv \bar{u}_r \bar{u}^2_s -[\bar{u}_r,
\bar{u}^2_s]=\bar{u}_r \bar{u}^2_s -2\bar{u}_s c_{rs}\pmod{\omega(L)^4}.\\
\end{split}
\]
Therefore the elements
$$
v_1=b_rb_s^2,\quad  v_2=b_s^2 b_r\quad  \text{and} \quad v_3=b_sb_rb_s
$$
are $F$-linearly dependent modulo $\omega(L)^4$.
In view of property (F-I),
$$ ({\mathfrak B}\cap
\omega(L)^3)\backslash \omega(L)^4
$$ is an $F$-basis for
$\omega(L)^3$ modulo $\omega(L)^4$. Consequently, it follows that
either $v_i\in \omega(L)^4$ for some $i\in \{1,2,3\}$ or
$v_j\equiv v_k(\textrm{mod}\, \omega(L)^4)$ for some $j,k \in
\{1,2,3\}$. In each case we have a contradiction to Lemma
\ref{L:1}, and the proof is complete.
\end{proof}

We remark that the previous result fails without the assumption on the characteristic of the ground field. Indeed, let $F$ be a field of characteristic 2 and consider the restricted Lie algebra $L=Fa+Fb+Fc$ with $[a,b]=c$, $[a,c]=[b,c]=0$, and $a^{[2]}=b^{[2]}=c^{[2]}=0$.
Then it is straightforward to show that
$$
\{1,a,b,ab,ab+c,ac,bc,abc\}
$$
is an f.m. basis of $u(L)$.

\smallskip

{\it Acknowledgement.} The authors are grateful to the referee for pointing out a problem
in the original proof of Theorem 1.

\newpage

\bibliographystyle{abbrv}
\bibliography{FMB}

\end{document}